\documentclass{amsart}
\usepackage{amsmath}
\usepackage{amsfonts}

\newtheorem{theorem}{Theorem}
\theoremstyle{plain}

\newtheorem{corollary}{Corollary}

\newtheorem{definition}{Definition}

\newtheorem{lemma}{Lemma}

\newtheorem{remark}{Remark}

\numberwithin{equation}{section}
\input{tcilatex}

\begin{document}
\title[Generalized Sz\'{a}sz-Mirakyan Operators]{$(p,q)$-Generalization of Sz%
\'{a}sz-Mirakyan Operators}
\author{Tuncer Acar}
\address{Kirikkale University, Faculty of Science and Arts, Department of
Mathematics, Yahsihan, 71450, Kirikkale, Turkey}
\email{tunceracar@ymail.com}

\begin{abstract}
In this paper, we introduce new modifications of Sz\'{a}sz-Mirakyan
operators based on $(p,q)$-integers. We first give a recurrence relation for
the moments of new operators and present explicit formula for the moments
and central moments up to order $4.$ Some approximation properties of new
operators are explored: the uniform convergence over bounded and unbounded
intervals is established, direct approximation properties of the operators
in terms of the moduli of smoothness is obtained and Voronovskaya theorem is
presented. For the particular case $p=1$, the previous results for $q$-Sz%
\'{a}sz-Mirakyan operators are captured.
\end{abstract}

\maketitle

\section{Introduction}

With a great potential for applications, approximation theory has an
important role in mathematical research. Since Korovkin's famous theorem in
1950, the study of the linear methods of approximation given by sequences of
positive and linear operators became a firmly entrenched part of
approximation theory. Due to this fact, the well-known operators as
Bernstein, Sz\'{a}sz, Baskakov etc. and their generalizations have been
studied intensively. The rapid development of $q$-calculus has led to the
discovery of new generalizations of approximation operators based on $q$%
-integers. A pioneer of $q$-calculus in approximation theory is Lupas \cite%
{lupas}, who first introduced the $q$-analogue of Bernstein polynomials.
Since approximation by $q$-Bernstein polynomials is better than classical
one under convenient choice of $q,$ generalizations of other well-known
approximation operators were introduced, we refer the readers to \cite{aral}
for more comprehensive details.

Very recently, Mursaleen et al. applied $(p,q)$-calculus in approximation
theory and introduced first $(p,q)$-analogue of Bernstein operators. They
investigated uniform convergence of the operators and rate of convergence,
obtained Voronovskaya theorem as well. Also, $(p,q)$-analogue of
Bernstein-Stancu operators and Bleimann-Butzer-Hahn operators were
introduced in \cite{mursalen2} and \cite{mursalen3}, respectively.

In the present paper, we introduce $(p,q)$-analogue of Sz\'{a}sz-Mirakyan
operators. We begin by recalling certain notations of $(p,q)$-calculus. Let $%
0<q<p\leq 1.$ For each nonnegative integer $k,$ $n$, $n\geq k\geq 0$, the $%
(p,q)$-integer $\left[ k\right] _{p,q}$, $(p,q)$-factorial $\left[ k\right]
_{p,q}!$ and $(p,q)$-binomial are defined by%
\begin{equation*}
\left[ k\right] _{p,q}:=\frac{p^{k}-q^{k}}{p-q},
\end{equation*}%
\begin{equation*}
\left[ k\right] _{p,q}!:=\left \{ 
\begin{array}{lll}
\left[ k\right] _{p,q}\left[ k-1\right] _{p,q}...1 & , & k\geq 1, \\ 
1, &  & k=0%
\end{array}%
\right. 
\end{equation*}%
and%
\begin{equation*}
\QATOPD[ ] {n}{k}_{p,q}:=\frac{\left[ n\right] _{p,q}!}{\left[ n-k\right]
_{p,q}!\left[ k\right] _{p,q}!},
\end{equation*}

Note that if we take $p=1$ in above notations, they reduce to $q$-analogues.

Further,%
\begin{equation*}
\left( x+y\right) _{p,q}^{n}:=\left( x+y\right) \left( px+qy\right) \left(
p^{2}x+q^{2}y\right) ...\left( p^{n-1}x+q^{n-1}y\right) .
\end{equation*}%
Also the $(p,q)$-derivative of a function $f$, denoted by $D_{p,q}f,$ is
defined by%
\begin{equation*}
\left( D_{p,q}f\right) \left( x\right) :=\frac{f\left( px\right) -f\left(
qx\right) }{\left( p-q\right) x},\text{ }x\neq 0,\text{ }\left(
D_{p,q}f\right) \left( 0\right) :=f^{\prime }\left( 0\right) 
\end{equation*}%
provided that $f$ is differentiable at $0.$ The formula for the $(p,q)$%
-derivative of a product is%
\begin{equation*}
D_{p,q}\left( u\left( x\right) v\left( x\right) \right) :=D_{p,q}\left(
u\left( x\right) \right) v\left( qx\right) +D_{p,q}\left( v\left( x\right)
\right) u\left( qx\right) .
\end{equation*}%
For more details on $(p,q)$-calculus, we refer the readers to \cite{Houn,
katriel, jag, sadjang, sahai} and references therein. There are two $(p,q)$%
-analogues of the exponential function, see \cite{jag},%
\begin{equation*}
e_{p,q}\left( x\right) =\dsum \limits_{n=0}^{\infty }\frac{p^{\frac{n\left(
n-1\right) }{2}}x^{n}}{\left[ n\right] _{p,q}!},
\end{equation*}%
and%
\begin{equation}
E_{p,q}\left( x\right) =\dsum \limits_{n=0}^{\infty }\frac{q^{\frac{n\left(
n-1\right) }{2}}x^{n}}{\left[ n\right] _{p,q}!},  \label{t11}
\end{equation}%
which satisfy the equality $e_{p,q}\left( x\right) E_{p,q}\left( -x\right)
=1.$ For $p=1$, $e_{p,q}\left( x\right) $ and $E_{p,q}\left( x\right) $
reduce to $q$-exponential functions.

\section{Construction of operators and Auxiliary results}

Considering the second $(p,q)$-exponential function (\ref{t11}), we set the $%
(p,q)$-Sz\'{a}sz-Mirakyan basis function as%
\begin{equation*}
s_{n}\left( p,q;x\right) =:\frac{1}{E\left( \left[ n\right] _{p,q}x\right) }%
q^{\frac{k\left( k-1\right) }{2}}\frac{\left[ n\right] _{p,q}^{k}x^{k}}{%
\left[ k\right] _{p,q}!},\text{ }n=1,2,...
\end{equation*}%
For $q\in \left( 0,1\right) ,$ $p\in \left( q,1\right] $ and $x\in \left[
0,\infty \right) $, $s_{n}\left( p,q;x\right) \geq 0$. We can also easily
see that%
\begin{equation*}
\sum_{k=0}^{\infty }s_{n}\left( p,q;x\right) =\frac{1}{E\left( \left[ n%
\right] _{p,q}x\right) }\sum_{k=0}^{\infty }q^{\frac{k\left( k-1\right) }{2}}%
\frac{\left[ n\right] _{p,q}^{k}x^{k}}{\left[ k\right] _{p,q}!}=1.
\end{equation*}

\begin{definition}
Let $0<q<p\leq 1$ and $n\in \mathbb{N}.$ For $f:\left[ 0,\infty \right)
\rightarrow \mathbb{R}$ we define the $(p,q)$-analogue of Sz\'{a}sz-Mirakyan
operators by%
\begin{equation}
S_{n,p,q}\left( f;x\right) :=\frac{1}{E\left( \left[ n\right] _{p,q}x\right) 
}\sum_{k=0}^{\infty }f\left( \frac{\left[ k\right] _{p,q}}{q^{k-2}\left[ n%
\right] _{p,q}}\right) q^{\frac{k\left( k-1\right) }{2}}\frac{\left[ n\right]
_{p,q}^{k}x^{k}}{\left[ k\right] _{p,q}!}.  \label{op}
\end{equation}
\end{definition}

The operators (\ref{op}) are linear and positive. For $p=1,$ the operators (%
\ref{op}) turn out to be $q$-Sz\'{a}sz-Mirakyan operators defined in \cite%
{mahmudov}.

\begin{lemma}
Let $0<q<p\leq 1$ and $n\in \mathbb{N}.$ We have%
\begin{equation}
S_{n,p,q}\left( t^{m+1};x\right) =\sum_{j=0}^{m}\binom{m}{j}\frac{xp^{j}}{%
q^{2j-m-1}\left[ n\right] _{p,q}^{m-j}}S_{n,p,q}\left( t^{j};x\right) .
\label{recurrence}
\end{equation}
\end{lemma}

\begin{proof}
Using the identity%
\begin{equation*}
\left[ k\right] _{p,q}=q^{k-1}+p\left[ k-1\right] _{p,q}
\end{equation*}%
we can write%
\begin{eqnarray*}
&&S_{n,p,q}\left( t^{m+1};x\right) \\
&=&\frac{1}{E\left( \left[ n\right] _{p,q}x\right) }\sum_{k=0}^{\infty }%
\frac{\left[ k\right] _{p,q}^{m+1}}{q^{\left( k-2\right) \left( m+1\right) }%
\left[ n\right] _{p,q}^{\left( m+1\right) }}q^{\frac{k\left( k-1\right) }{2}}%
\frac{\left[ n\right] _{p,q}^{k}x^{k}}{\left[ k\right] _{p,q}!} \\
&=&\frac{q}{E\left( \left[ n\right] _{p,q}x\right) }\sum_{k=1}^{\infty }%
\frac{\left[ k\right] _{p,q}^{m}}{q^{\left( k-2\right) }\left[ n\right]
_{p,q}^{m}}q^{\frac{k\left( k-1\right) }{2}-k+1}\frac{\left[ n\right]
_{p,q}^{k-1}x^{k}}{\left[ k-1\right] _{p,q}!} \\
&=&\frac{q}{E\left( \left[ n\right] _{p,q}x\right) }\sum_{k=1}^{\infty }%
\frac{\left( q^{k-1}+p\left[ k-1\right] _{p,q}\right) ^{m}}{q^{\left(
k-2\right) m}\left[ n\right] _{p,q}^{m}}q^{\frac{k\left( k-1\right) }{2}-k+1}%
\frac{\left[ n\right] _{p,q}^{k-1}x^{k}}{\left[ k-1\right] _{p,q}!} \\
&=&\frac{q}{E\left( \left[ n\right] _{p,q}x\right) }\sum_{k=1}^{\infty
}\sum_{j=0}^{m}\binom{m}{j}q^{\left( k-1\right) \left( m-j\right) }p^{j}%
\left[ k-1\right] _{p,q}^{j}\frac{q^{\frac{\left( k-1\right) \left(
k-2\right) }{2}}}{q^{\left( k-2\right) m}\left[ n\right] _{p,q}^{m}}\frac{%
\left[ n\right] _{p,q}^{k-1}x^{k}}{\left[ k-1\right] _{p,q}!} \\
&=&\frac{q}{E\left( \left[ n\right] _{p,q}x\right) }\sum_{j=0}^{m}\binom{m}{j%
}\frac{xp^{j}}{q^{2j-m}\left[ n\right] _{p,q}^{m-j}}\sum_{k=1}^{\infty }%
\frac{\left[ k-1\right] _{p,q}^{j}}{q^{\left( k-3\right) j}\left[ n\right]
_{p,q}^{j}}q^{\frac{\left( k-1\right) \left( k-2\right) }{2}}\frac{\left[ n%
\right] _{p,q}^{k-1}x^{k-1}}{\left[ k-1\right] _{p,q}!} \\
&=&\frac{q}{E\left( \left[ n\right] _{p,q}x\right) }\sum_{j=0}^{m}\binom{m}{j%
}\frac{xp^{j}}{q^{2j-m}\left[ n\right] _{p,q}^{m-j}}\sum_{k=0}^{\infty }%
\frac{\left[ k\right] _{p,q}^{j}}{q^{\left( k-2\right) j}\left[ n\right]
_{p,q}^{j}}q^{\frac{k\left( k-1\right) }{2}}\frac{\left[ n\right]
_{p,q}^{k}x^{k}}{\left[ k\right] _{p,q}!} \\
&=&\sum_{j=0}^{m}\binom{m}{j}\frac{xp^{j}}{q^{2j-m-1}\left[ n\right]
_{p,q}^{m-j}}S_{n,p,q}\left( t^{j};x\right) ,
\end{eqnarray*}%
which is desired.
\end{proof}

\begin{lemma}
\label{MOMENTS}Let $0<q<p\leq 1$ and $n\in \mathbb{N}.$ We have%
\begin{eqnarray}
S_{n,p,q}\left( 1;x\right)  &=&1  \label{mom1} \\
S_{n,p,q}\left( t;x\right)  &=&qx  \label{mom2} \\
S_{n,p,q}\left( t^{2};x\right)  &=&pqx^{2}+\frac{q^{2}x}{\left[ n\right]
_{p,q}}  \label{mom3} \\
S_{n,p,q}\left( t^{3};x\right)  &=&p^{3}x^{3}+\frac{x^{2}\left(
p^{2}q+2pq^{2}\right) }{\left[ n\right] _{p,q}}+\frac{q^{3}x}{\left[ n\right]
_{p,q}^{2}}  \notag \\
S_{n,p,q}\left( t^{4};x\right)  &=&\frac{p^{6}x^{4}}{q^{2}}+\frac{x^{3}}{%
\left[ n\right] _{p,q}}p^{3}q\left( \frac{p^{2}+2q+3q^{2}}{q^{2}}\right)  
\notag \\
&&+\frac{x^{2}}{\left[ n\right] _{p,q}^{2}}pq\left( p^{2}+3pq+3q^{2}\right) +%
\frac{q^{4}x}{\left[ n\right] _{p,q}^{3}}.  \notag
\end{eqnarray}
\end{lemma}

\begin{proof}
Since the proof of each equalities have same method we give the proof for
only last two equalities. Using (\ref{recurrence}), we get%
\begin{eqnarray*}
S_{n,p,q}\left( t^{3};x\right)  &=&\frac{q^{3}x}{\left[ n\right] _{p,q}^{2}}%
S_{n,p,q}\left( 1;x\right) +\frac{2xpq}{\left[ n\right] _{p,q}}%
S_{n,p,q}\left( t;x\right) +\frac{xp^{2}}{q}S_{n,p,q}\left( t^{2};x\right) 
\\
&=&\frac{q^{3}x}{\left[ n\right] _{p,q}^{2}}+\frac{2x^{2}pq^{2}}{\left[ n%
\right] _{p,q}}+\frac{xp^{2}}{q}\left( pqx^{2}+\frac{q^{2}x}{\left[ n\right]
_{p,q}}\right)  \\
&=&p^{3}x^{3}+\frac{x^{2}\left( p^{2}q+2pq^{2}\right) }{\left[ n\right]
_{p,q}}+\frac{q^{3}x}{\left[ n\right] _{p,q}^{2}}
\end{eqnarray*}%
and%
\begin{eqnarray*}
S_{n,p,q}\left( t^{4};x\right)  &=&\frac{q^{4}x}{\left[ n\right] _{p,q}^{3}}%
S_{n,p,q}\left( 1;x\right) +\frac{3xpq^{2}}{\left[ n\right] _{p,q}^{2}}%
S_{n,p,q}\left( t;x\right)  \\
&&+\frac{3xp^{2}}{\left[ n\right] _{p,q}}S_{n,p,q}\left( t^{2};x\right) +%
\frac{xp^{3}}{q^{2}}S_{n,p,q}\left( t^{3};x\right)  \\
&=&\frac{q^{4}x}{\left[ n\right] _{p,q}^{3}}+\frac{3x^{2}pq^{3}}{\left[ n%
\right] _{p,q}^{2}}+\frac{3xp^{2}}{\left[ n\right] _{p,q}}\left( pqx^{2}+%
\frac{q^{2}x}{\left[ n\right] _{p,q}}\right)  \\
&&+\frac{xp^{3}}{q^{2}}\left( p^{3}x^{3}+\frac{x^{2}\left(
p^{2}q+2pq^{2}\right) }{\left[ n\right] _{p,q}}+\frac{q^{3}x}{\left[ n\right]
_{p,q}^{2}}\right)  \\
&=&\frac{p^{6}x^{4}}{q^{2}}+\frac{x^{3}}{\left[ n\right] _{p,q}}p^{3}q\left( 
\frac{p^{2}+2q+3q^{2}}{q^{2}}\right)  \\
&&+\frac{x^{2}}{\left[ n\right] _{p,q}^{2}}pq\left( p^{2}+3pq+3q^{2}\right) +%
\frac{q^{4}x}{\left[ n\right] _{p,q}^{3}}.
\end{eqnarray*}
\end{proof}

\begin{corollary}
Using Lemma \ref{MOMENTS}, we immediately have the following explicit
formulas for the central moments.%
\begin{align}
S_{n,p,q}\left( t-x;x\right) & =\left( q-1\right) x  \notag \\
S_{n,p,q}\left( \left( t-x\right) ^{2};x\right) & =x^{2}\left(
pq-2q+1\right) +\frac{q^{2}x}{\left[ n\right] _{p,q}}  \label{second mom} \\
S_{n,p,q}\left( \left( t-x\right) ^{4};x\right) & =x^{4}\left( \frac{p^{6}}{%
q^{2}}-4p^{3}+6pq-4q+1\right)   \notag \\
& +\frac{x^{3}}{\left[ n\right] _{p,q}}\left( p^{3}q\left( \frac{%
p^{2}+2q+3q^{2}}{q^{2}}\right) -4\left( p^{2}q+2pq^{2}\right) +6q^{2}\right) 
\notag \\
& +\frac{x^{2}}{\left[ n\right] _{p,q}^{2}}\left( pq\left(
p^{2}+3pq+3q^{2}\right) -4q^{3}\right) +\frac{x}{\left[ n\right] _{p,q}^{3}}%
q^{4}  \label{fourth mom}
\end{align}
\end{corollary}

\begin{remark}
\label{remark}For $q\in \left( 0,1\right) $ and $p\in \left( q,1\right] $ we
easily see that $\lim_{n\rightarrow \infty }\left[ n\right] _{p,q}=1/\left(
p-q\right) .$ Hence, the operators (\ref{op}) are not approximation process
with above form. In order to study convergence properties of the sequence of
($p,q)$-Sz\'{a}sz operators, we assume that $q=\left( q_{n}\right) $\ and $%
p=\left( p_{n}\right) $ such that $0<q_{n}<p_{n}\leq 1$ and $%
q_{n}\rightarrow 1,$ $p_{n}\rightarrow 1,$ $q_{n}^{n}\rightarrow a,$ $%
p_{n}^{n}\rightarrow b$ as $n\rightarrow \infty .$ We also assume that%
\begin{eqnarray*}
\lim_{n\rightarrow \infty }\left[ n\right] _{p_{n},q_{n}}\left( q-1\right) 
&=&\alpha , \\
\lim_{n\rightarrow \infty }\left[ n\right] _{p_{n},q_{n}}\left(
p_{n}q_{n}-2q_{n}+1\right)  &=&\gamma , \\
\lim_{n\rightarrow \infty }\left[ n\right] _{p_{n},q_{n}}\left( \frac{%
p_{n}^{6}}{q_{n}^{2}}-4p_{n}^{3}+6p_{n}q_{n}-4q_{n}+1\right)  &=&\beta ,
\end{eqnarray*}%
It is natural to ask the existence of such sequences $\left( q_{n}\right) $
and $\left( p_{n}\right) .$ For example, let $c,d\in \mathbb{R}^{+}$ such
that $c>d.$ If we choose $q_{n}=\frac{n}{n+c}$ and $p_{n}=\frac{n}{n+d}$
then $q_{n}\rightarrow 1,$ $p_{n}\rightarrow 1,q_{n}^{n}\rightarrow e^{-c},$ 
$p_{n}^{n}\rightarrow e^{-d},$ $\lim_{n\rightarrow \infty }\left[ n\right]
_{p,q}=\infty $ as $n\rightarrow \infty .$ Also we have $\alpha =\frac{%
a\left( e^{-d}-e^{-c}\right) }{d-c},$ $\gamma =e^{-d}-e^{-c},$ $\beta =0.$
\end{remark}

\begin{corollary}
According to Remark \ref{remark}, we immediately have%
\begin{eqnarray}
\lim_{n\rightarrow \infty }\left[ n\right] _{p_{n},q_{n}}S_{n,p_{n},q_{n}}%
\left( t-x;x\right) &=&\alpha x,  \label{s1} \\
\lim_{n\rightarrow \infty }\left[ n\right] _{p_{n},q_{n}}S_{n,p_{n},q_{n}}%
\left( \left( t-x\right) ^{2};x\right) &=&\gamma x^{2}+x,  \label{s2} \\
\lim_{n\rightarrow \infty }\left[ n\right] _{p_{n},q_{n}}S_{n,p_{n},q_{n}}%
\left( \left( t-x\right) ^{4};x\right) &=&\beta x^{4}.  \label{r1}
\end{eqnarray}
\end{corollary}

\section{Direct Results}

In this section, we present local approximation theorem for the operators $%
S_{n,p,q}.$ By $C_{B}\left[ 0,\infty \right) ,$ we denote the space of
real-valued continuous and bounded functions $f$ defined on the interval $%
\left[ 0,\infty \right) .$ The norm $\left \Vert \cdot \right \Vert $ on the
space $C_{B}\left[ 0,\infty \right) $ is given by%
\begin{equation*}
\left \Vert f\right \Vert =\sup_{0\leq x<\infty }\left \vert f\left(
x\right) \right \vert .
\end{equation*}%
Further let us consider the following $\mathcal{K}$-functional:%
\begin{equation*}
\mathcal{K}_{2}\left( f,\delta \right) =\inf_{g\in W^{2}}\left \{ \left
\Vert f-g\right \Vert +\delta \left \Vert g^{\prime \prime }\right \Vert
\right \} ,
\end{equation*}%
where $\delta >0$ and $W^{2}=\left \{ g\in C_{B}\left[ 0,\infty \right)
:g^{\prime },g^{\prime \prime }\in C_{B}\left[ 0,\infty \right) \right \} .$
By \cite[p. 177, Theorem 2.4]{devore} there exists an absolute constant $C>0$
such that%
\begin{equation}
\mathcal{K}_{2}\left( f,\delta \right) \leq C\omega _{2}\left( f,\sqrt{%
\delta }\right) ,  \label{eq21}
\end{equation}%
where%
\begin{equation*}
\omega _{2}\left( f,\delta \right) =\sup_{0<h\leq \sqrt{\delta }}\sup_{x\in %
\left[ 0,\infty \right) }\left \vert f\left( x+2h\right) -2f\left(
x+h\right) +f\left( x\right) \right \vert
\end{equation*}%
is the second order modulus of smoothness of $f\in C_{B}\left[ 0,\infty
\right) .$ The usual modulus of continuity of $f\in C_{B}\left[ 0,\infty
\right) $ is defined by%
\begin{equation*}
\omega \left( f,\delta \right) =\sup_{0<h\leq \delta }\sup_{x\in \left[
0,\infty \right) }\left \vert f\left( x+h\right) -f\left( x\right) \right
\vert .
\end{equation*}

\begin{theorem}
Let $p,q\in \left( 0,1\right) $ such that $0<q<p\leq 1.$ Then we have%
\begin{equation*}
\left \vert S_{n,p,q}\left( f;x\right) -f\left( x\right) \right \vert \leq
M\omega _{2}\left( f,\sqrt{\delta _{n}\left( x\right) }\right) +\omega
\left( f,\left( 1-q\right) x\right) ,
\end{equation*}%
for every $x\in \left[ 0,\infty \right) $ and $f\in C_{B}\left[ 0,\infty
\right) ,$ where $\delta _{n}\left( x\right) =2x^{2}\left( 1-q\right) +\frac{%
x}{\left[ n\right] _{p,q}}.$
\end{theorem}

\begin{proof}
Let us consider the auxiliary operator%
\begin{equation}
\hat{S}_{n,p,q}\left( f;x\right) =S_{n,p,q}\left( f;x\right) +f\left(
x\right) -f\left( qx\right) .  \label{acar1}
\end{equation}%
The operators $\hat{S}_{n,p,q}$ are linear and it is clear by Lemma \ref%
{MOMENTS} that%
\begin{equation}
\hat{S}_{n,p,q}\left( 1;x\right) =S_{n,p,q}\left( 1;x\right) =1  \label{12}
\end{equation}%
and%
\begin{equation}
S_{n,p,q}\left( t;x\right) =S_{n,p,q}\left( t;x\right) +x-qx=x.  \label{13}
\end{equation}%
Let $g\in W^{2}.$ The classical Taylor's expansion of $g\in \mathcal{W}%
_{\infty }^{2}$ yields for $t\in 
\mathbb{R}
^{+}$ that%
\begin{eqnarray*}
g\left( t\right) &=&g\left( x\right) +g^{\prime }\left( x\right) \left(
t-x\right) \\
&&+\int \limits_{x}^{t}\left( t-u\right) g^{\prime \prime }\left( u\right)
du.
\end{eqnarray*}%
Using (\ref{12}) and (\ref{13}) we have%
\begin{equation}
\hat{S}_{n,p,q}\left( g;x\right) -g\left( x\right) =\hat{S}_{n,p,q}\left(
\int \limits_{x}^{t}\left( t-u\right) g^{\prime \prime }\left( u\right)
du;x\right) .  \label{14}
\end{equation}%
Hence, by (\ref{acar1}) we can write$\hat{S}_{n,p,q}\left( f;x\right)
=S_{n,p,q}\left( f;x\right) +f\left( x\right) -f\left( qx\right) $%
\begin{eqnarray}
\left \vert \hat{S}_{n,p,q}\left( g;x\right) -g\left( x\right) \right \vert
&\leq &\left \vert S_{n,p,q}\left( \int \limits_{x}^{t}\left( t-u\right)
g^{\prime \prime }\left( u\right) du;x\right) \right \vert  \notag \\
&&+S_{n,p,q}\left( \int \limits_{x}^{qx}\left \vert qx-u\right \vert
g^{\prime \prime }\left( u\right) du;x\right)  \notag \\
&\leq &\left \Vert g^{\prime \prime }\right \Vert \left[ S_{n,p,q}\left(
\left( t-x\right) ^{2};x\right) +\left( qx-x\right) ^{2}\right] .
\label{acar2}
\end{eqnarray}%
On the other hand, using (\ref{second mom}) we get%
\begin{eqnarray}
S_{n,p,q}\left( \left( t-x\right) ^{2};x\right) &=&x^{2}\left(
pq-2q+1\right) +\frac{q^{2}x}{\left[ n\right] _{p,q}}  \notag \\
&=&x^{2}\left( \left( p-1\right) q+1-q\right) +\frac{q^{2}x}{\left[ n\right]
_{p,q}}  \notag \\
&\leq &x^{2}\left( \left( 1-p\right) q+1-q\right) +\frac{q^{2}x}{\left[ n%
\right] _{p,q}}  \notag \\
&\leq &x^{2}\left( 1-pq\right) +\frac{x}{\left[ n\right] _{p,q}}.
\label{acar3}
\end{eqnarray}%
Then, by (\ref{acar2}) we have%
\begin{eqnarray*}
\left \vert \hat{S}_{n,p,q}\left( g;x\right) -g\left( x\right) \right \vert
&\leq &\left \Vert g^{\prime \prime }\right \Vert \left[ x^{2}\left(
1-pq\right) +\frac{x}{\left[ n\right] _{p,q}}+\left( qx-x\right) ^{2}\right]
\\
&\leq &\left \Vert g^{\prime \prime }\right \Vert \left[ 2x^{2}\left(
1-q\right) +\frac{x}{\left[ n\right] _{p,q}}\right] .
\end{eqnarray*}%
Also, using (\ref{acar1}), (\ref{op}) and (\ref{mom1}) we obtain%
\begin{equation*}
\left \vert \hat{S}_{n,p,q}\left( f;x\right) \right \vert \leq \left \vert
S_{n,p,q}\left( f;x\right) \right \vert +2\left \Vert f\right \Vert \leq
\left \Vert f\right \Vert S_{n,p,q}\left( 1;x\right) +2\left \Vert f\right
\Vert \leq 3\left \Vert f\right \Vert .
\end{equation*}%
Observe that%
\begin{eqnarray*}
\left \vert S_{n,p,q}\left( f;x\right) -f\left( x\right) \right \vert &\leq
&\left \vert \hat{S}_{n,p,q}\left( f-g;x\right) -\left( f-g\right) \left(
x\right) \right \vert \\
&&+\left \vert \hat{S}_{n,p,q}\left( g;x\right) -g\left( x\right) \right
\vert \\
&&+\left \vert f\left( x\right) -f\left( qx\right) \right \vert \\
&\leq &4\left \Vert f-g\right \Vert +\left \Vert g^{\prime \prime }\right
\Vert \left( 2x^{2}\left( 1-q\right) +\frac{x}{\left[ n\right] _{p,q}}\right)
\\
&&+\omega \left( f,\left( 1-q\right) x\right) .
\end{eqnarray*}%
Taking infimum on the right hand side over all $g\in \mathcal{W}_{\infty
}^{2}$ we have%
\begin{eqnarray*}
\left \vert S_{n,p,q}\left( f;x\right) -f\left( x\right) \right \vert &\leq
&4\mathcal{K}_{2}\left( f,2x^{2}\left( 1-q\right) +\frac{x}{\left[ n\right]
_{p,q}}\right) \\
&&+\omega \left( f,\left( 1-q\right) x\right) .
\end{eqnarray*}%
Using equality (\ref{eq21}) for every $q\in \left( 0,1\right) $ and $p\in
\left( q,1\right] $ we get 
\begin{eqnarray*}
\left \vert S_{n,p,q}\left( f;x\right) -f\left( x\right) \right \vert &\leq
&M\omega _{2}\left( f,\sqrt{2x^{2}\left( 1-q\right) +\frac{x}{\left[ n\right]
_{p,q}}}\right) \\
&&+\omega \left( f,\left( 1-q\right) x\right) ,
\end{eqnarray*}%
which completes the proof.
\end{proof}

\section{Rate of Convergence}

First, let us recall the definitions of weighted spaces and corresponding
modulus of continuity. Let $C\left[ 0,\infty \right) $ be the set of all
continuous functions $f$ defined on $\left[ 0,\infty \right) $ and $B_{2}%
\left[ 0,\infty \right) $ the set of all functions $f$ defined on $\left[
0,\infty \right) $ satisfying the condition $\left \vert f\left( x\right)
\right \vert \leq M\left( 1+x^{2}\right) $ with some positive constant $M$
which may depend only on $f.$ $C_{2}\left[ 0,\infty \right) $ denotes the
subspace of all continuous functions in $B_{2}\left[ 0,\infty \right) .$\ By 
$C_{2}^{\ast }\left[ 0,\infty \right) $, we denote the subspace of all
functions $f\in C_{2}\left[ 0,\infty \right) $ for which $\lim_{x\rightarrow
\infty }\frac{f\left( x\right) }{1+x^{2}}$ is finite$.$ $B_{2}\left[
0,\infty \right) $ is a linear normed space with the norm $\left \Vert
f\right \Vert _{2}=\sup_{x\geq 0}\frac{\left \vert f\left( x\right) \right
\vert }{1+x^{2}}.$ For any positive $a,$ by%
\begin{equation*}
\omega _{a}\left( f,\delta \right) =\sup_{\left \vert t-x\right \vert \leq
\delta }\sup_{x,t\in \left[ 0,a\right] }\left \vert f\left( t\right)
-f\left( x\right) \right \vert
\end{equation*}%
we denote the usual modulus of continuity of $f$ on the closed interval $%
\left[ 0,a\right] .$ For a function $f\in C_{2}\left[ 0,\infty \right) $,
the modulus of continuity $\omega _{a}\left( f,\delta \right) $ tends to
zero.

\begin{theorem}
Let $f\in C_{2}^{\ast }\left[ 0,\infty \right) ,$ $q=q_{n}\in \left(
0,1\right) $, $p=p_{n}\in \left( q,1\right] $ such that $q_{n}\rightarrow 1,$
$p_{n}\rightarrow 1$ as $n\rightarrow \infty $ and $\omega _{a+1}\left(
f,\delta \right) $ be its modulus of continuity on the finite interval $%
\left[ 0,a+1\right] \subset \left[ 0,\infty \right) .$ Then we have%
\begin{eqnarray*}
\left \Vert S_{n,p,q}\left( f;x\right) -f\left( x\right) \right \Vert _{C
\left[ 0,a\right] } &\leq &6M_{f}\left( 1+a^{2}\right) a\left( a\left(
1-pq\right) +\frac{1}{\left[ n\right] _{p,q}}\right) \\
&&+2\omega _{a+1}\left( f,a^{2}\left( 1-pq\right) +\frac{a}{\left[ n\right]
_{p,q}}\right) .
\end{eqnarray*}
\end{theorem}

\begin{proof}
For $x\in \left[ 0,a\right] $ and $t>a+1,$ since $t-x>1$, we have%
\begin{eqnarray}
\left \vert f\left( t\right) -f\left( x\right) \right \vert &\leq
&M_{f}\left( 2+x^{2}+t^{2}\right)  \notag \\
&\leq &M_{f}\left( 2+3x^{2}+2\left( t-x\right) ^{2}\right)  \notag \\
&\leq &6M_{f}\left( 1+a^{2}\right) \left( t-x\right) ^{2}.  \label{a1}
\end{eqnarray}%
For $x\in \left[ 0,a\right] $ and $t\leq a+1,$ we have%
\begin{equation}
\left \vert f\left( t\right) -f\left( x\right) \right \vert \leq \omega
_{a+1}\left( f,\left \vert t-x\right \vert \right) \leq \left( 1+\frac{\left
\vert t-x\right \vert }{\delta }\right) \omega _{a+1}\left( f,\delta \right)
,\text{ }\delta >0.  \label{a2}
\end{equation}%
By (\ref{a1}) and (\ref{a2}), we can write%
\begin{equation*}
\left \vert f\left( t\right) -f\left( x\right) \right \vert \leq
6M_{f}\left( 1+a^{2}\right) \left( t-x\right) ^{2}+\left( 1+\frac{\left
\vert t-x\right \vert }{\delta }\right) \omega _{a+1}\left( f,\delta \right)
\end{equation*}%
for $x\in \left[ 0,a\right] $ and $t>0.$ Thus, applying the operators to
both sides of above inequality, we have%
\begin{eqnarray*}
\left \vert S_{n,p,q}\left( f;x\right) -f\left( x\right) \right \vert &\leq
&S_{n,p,q}\left( \left \vert f\left( t\right) -f\left( x\right) \right \vert
;x\right) \\
&\leq &6M_{f}\left( 1+a^{2}\right) S_{n,p,q}\left( \left( t-x\right)
^{2};x\right) \\
&&+\left( 1+\frac{\left \vert t-x\right \vert }{\delta }\right) \omega
_{a+1}\left( f,\delta \right) \left( 1+\frac{1}{\delta }S_{n,p,q}\left(
\left( t-x\right) ^{2};x\right) \right) ^{1/2}.
\end{eqnarray*}%
Hence, by Cauchy-Schwarz inequality and inequality (\ref{acar3}) we obtain$%
x^{2}\left( 1-pq\right) +\frac{x}{\left[ n\right] _{p,q}}$%
\begin{eqnarray*}
\left \vert S_{n,p,q}\left( f;x\right) -f\left( x\right) \right \vert &\leq
&6M_{f}\left( 1+a^{2}\right) \left[ x^{2}\left( 1-pq\right) +\frac{x}{\left[
n\right] _{p,q}}\right] \\
&&+\omega _{a+1}\left( f,\delta \right) \left( 1+\frac{1}{\delta }\left(
x^{2}\left( 1-pq\right) +\frac{x}{\left[ n\right] _{p,q}}\right) \right)
^{1/2} \\
&\leq &6M_{f}\left( 1+a^{2}\right) a\left( a\left( 1-pq\right) +\frac{1}{%
\left[ n\right] _{p,q}}\right) \\
&&+\omega _{a+1}\left( f,\delta \right) \left( 1+\frac{1}{\delta }\left(
a^{2}\left( 1-pq\right) +\frac{a}{\left[ n\right] _{p,q}}\right)
^{1/2}\right) .
\end{eqnarray*}%
Choosing $\delta =\sqrt{a^{2}\left( 1-pq\right) +\frac{a}{\left[ n\right]
_{p,q}}}$ we have desired result
\end{proof}

\section{Weighted Approximation By $S_{n,p,q}$}

Now we give approximation properties of the operators $S_{n,p,q}$ on the
interval $\left[ 0,\infty \right) $.

Since%
\begin{eqnarray*}
S_{n,p,q}\left( 1+t^{2};x\right)  &=&1+pqx^{2}+\frac{q^{2}x}{\left[ n\right]
_{p,q}} \\
&\leq &1+x^{2}+x
\end{eqnarray*}%
and if $x\in \left[ 0,1\right] ,$ then $x\leq 1$ and if $x\in \left(
1,\infty \right) ,$ then $x\leq x^{2}.$ Hence we have%
\begin{equation*}
S_{n,p,q}\left( 1+t^{2};x\right) \leq 2\left( 1+x^{2}\right) 
\end{equation*}%
which says that $S_{n,p,q}$ are linear positive operators acting from $C_{2}%
\left[ 0,\infty \right) $ to $B_{2}\left[ 0,\infty \right) .$ For more
details see \cite{Gadjiev1974, Gadjiev 1976}.

\textbf{Theorem. A. }Let the sequence of linear positive operators $\left(
L_{n}\right) $ acting from $C_{2}\left[ 0,\infty \right) $ to $B_{2}\left[
0,\infty \right) $ satisfy the three conditions%
\begin{equation*}
\lim_{n\rightarrow \infty }\left \Vert L_{n}e_{i}-e_{i}\right \Vert _{2}=0,%
\text{ }i=0,1,2.
\end{equation*}%
Then for any function $f\in C_{2}^{\ast }\left[ 0,\infty \right) $%
\begin{equation*}
\lim_{n\rightarrow \infty }\left \Vert L_{n}f-f\right \Vert _{2}=0.
\end{equation*}

\begin{theorem}
\label{uniform theo}Let $q=q_{n}\in \left( 0,1\right) $, $p=p_{n}\in \left(
q,1\right] $ such that $q_{n}\rightarrow 1,$ $p_{n}\rightarrow 1$ as $%
n\rightarrow \infty .$ Then for each function $f\in C_{2}^{\ast }\left[
0,\infty \right) $ we get%
\begin{equation*}
\lim_{n\rightarrow \infty }\left \Vert S_{n,p_{n},q_{n}}f-f\right \Vert
_{2}=0.
\end{equation*}
\end{theorem}

\begin{proof}
According to Theorem A, it is sufficient to verify the following three
conditions%
\begin{equation}
\lim_{n\rightarrow \infty }\left \Vert S_{n,p_{n},q_{n}}e_{i}-e_{i}\right
\Vert _{2}=0,\text{ }i=0,1,2.  \label{t1}
\end{equation}%
By (\ref{mom1}), (\ref{t1}) holds for $i=0.$ By (\ref{mom2}) and (\ref{mom3}%
) we have%
\begin{equation*}
\left \Vert S_{n,p_{n},q_{n}}e_{1}-e_{1}\right \Vert _{2}=\sup_{x\geq 0}%
\frac{\left( 1-q_{n}\right) x}{1+x^{2}}\leq 1-q_{n}
\end{equation*}%
and%
\begin{equation*}
\left \Vert S_{n,p_{n},q_{n}}e_{2}-e_{2}\right \Vert _{2}=\sup_{x\geq 0}%
\frac{\left( 1-p_{n}q_{n}\right) x^{2}+\frac{q_{n}^{2}x}{\left[ n\right]
_{p_{n},qn}}}{1+x^{2}}\leq \left( 1-p_{n}q_{n}\right) +\frac{1}{\left[ n%
\right] _{p_{n},q_{n}}}.
\end{equation*}%
Last two inequality mean that (\ref{t1}) holds for $i=1,2.$ By Theorem A,
the proof is completed.
\end{proof}

The weighted modulus of continuity is denoted by $\Omega \left( f;\delta
\right) $ and given by 
\begin{equation}
\Omega \left( f;\delta \right) =\sup_{0\leq h<\delta ,\text{ }x\in \left[
0,\infty \right) }\frac{\left \vert f\left( x+h\right) -f\left( x\right)
\right \vert }{\left( 1+h^{2}\right) \left( 1+x^{2}\right) }  \label{weight1}
\end{equation}%
for $f\in C_{2}\left[ 0,\infty \right) $. We know that for every $f\in
C_{2}^{\ast }\left[ 0,\infty \right) ,$ $\Omega \left( .;\delta \right) $%
\textbf{\ }has the properties

\begin{equation*}
\lim_{\delta \rightarrow 0}\Omega \left( f;\delta \right) =0
\end{equation*}%
and 
\begin{equation}
\Omega \left( f;\lambda \delta \right) \leq 2\left( 1+\lambda \right) \left(
1+\delta ^{2}\right) \Omega \left( f;\delta \right) ,\text{ \ }\lambda >0.
\label{weight2}
\end{equation}%
For $f\in C_{2}\left[ 0,\infty \right) ,$ from (\ref{weight1}) and (\ref%
{weight2}) we can write

\begin{align}
\left \vert f\left( t\right) -f\left( x\right) \right \vert & \leq \left(
1+\left( t-x\right) ^{2}\right) \left( 1+x^{2}\right) \Omega \left(
f;\left \vert t-x\right \vert \right)   \label{weight3} \\
& \leq 2\left( 1+\frac{\left \vert t-x\right \vert }{\delta }\right) \left(
1+\delta ^{2}\right) \Omega \left( f;\delta \right) \left( 1+\left(
t-x\right) ^{2}\right) \left( 1+x^{2}\right) .  \notag
\end{align}%
All concepts mentioned above can be found in \cite{nurhayat}.

\begin{theorem}
Let $0<q=q_{n}<p=p_{n}\leq 1$ such that $q_{n}\rightarrow 1,$ $%
p_{n}\rightarrow 1$ as $n\rightarrow \infty .$ Then for $f\in C_{2}^{\ast }%
\left[ 0,\infty \right) $ there exists a positive constant $A$ such that the
inequality 
\begin{equation*}
\sup_{x\in \left[ 0,\infty \right) }\frac{\left \vert S_{n,p,q}\left(
f;x\right) -f\left( x\right) \right \vert }{\left( 1+x^{2}\right) ^{5/2}}\leq
A\Omega \left( f;1/\sqrt{\beta _{p,q}\left( n\right) }\right) 
\end{equation*}%
holds, where $\beta _{p,q}\left( n\right) =\max \left \{ 1-pq,1/\left[ n%
\right] _{p,q}\right \} $ and $A$ is a positive constant.
\end{theorem}

\begin{proof}
Since $S_{n,p,q}\left( 1;x\right) =1$ and using the monotonicity of $%
S_{n,p,q}$ we can write%
\begin{equation*}
\left \vert S_{n,p,q}\left( f;x\right) -f\left( x\right) \right \vert \leq
S_{n,p,q}\left( \left \vert f\left( t\right) -f\left( x\right) \right \vert
;x\right) .
\end{equation*}%
On the other hand, we have from (\ref{weight3}) that%
\begin{eqnarray*}
&&\left \vert S_{n,p,q}\left( f;x\right) -f(x)\right \vert  \\
&\leq &2\left( 1+\delta ^{2}\right) \Omega \left( f;\delta \right) \left(
1+x^{2}\right) \left[ S_{n,p,q}\left( \left( 1+\frac{\left \vert
t-x\right \vert }{\delta }\right) \left( 1+\left( t-x\right) ^{2}\right)
;x\right) \right]  \\
&\leq &2\left( 1+\delta ^{2}\right) \Omega \left( f;\delta \right) \left(
1+x^{2}\right) \left \{ S_{n,p,q}\left( 1;x\right) +S_{n,p,q}\left( \left(
t-x\right) ^{2};x\right) \right.  \\
&&\left. +\frac{1}{\delta }S_{n,p,q}\left( \left \vert t-x\right \vert
;x\right) +\frac{1}{\delta }S_{n,p,q}\left( \left \vert t-x\right \vert \left(
t-x\right) ^{2};x\right) \right \} .
\end{eqnarray*}%
Using Cauchy-Schwarz inequality, we can write%
\begin{eqnarray*}
\left \vert S_{n,p,q}\left( f;x\right) -f\left( x\right) \right \vert  &\leq
&2\left( 1+\delta ^{2}\right) \Omega \left( f;\delta \right) \left(
1+x^{2}\right) \left \{ S_{n,p,q}\left( 1;x\right) +S_{n,p,q}\left( \left(
t-x\right) ^{2};x\right) \right.  \\
&&\left. +\frac{1}{\delta }\sqrt{S_{n,p,q}\left( \left( t-x\right)
^{2};x\right) }+\frac{1}{\delta }\sqrt{S_{n,p,q}\left( \left( t-x\right)
^{4};x\right) }\sqrt{S_{n,p,q}\left( \left( t-x\right) ^{2};x\right) }%
\right \} .
\end{eqnarray*}%
On the other hand, using (\ref{second mom}) we have%
\begin{eqnarray*}
S_{n,p,q}\left( \left( t-x\right) ^{2};x\right)  &\leq &x^{2}\left(
1-pq\right) +\frac{x}{\left[ n\right] _{p,q}} \\
&\leq &C_{1}\mathcal{O}\left( \beta _{p,q}\left( n\right) \right) \left(
1+x^{2}\right) ,
\end{eqnarray*}%
where $C_{1}>0$ and $\beta _{p,q}\left( n\right) =\max \left \{ 1-pq,1/\left[
n\right] _{p,q}\right \} .$ Since $\lim_{n\rightarrow \infty }p_{n}q_{n}=1$ $%
\lim_{n\rightarrow \infty }1/\left[ n\right] _{p_{n},q_{n}}=0,$ there exists
a positive constant $A_{2}$ such that%
\begin{equation*}
S_{n,p,q}\left( \left( t-x\right) ^{2};x\right) \leq A_{2}\left(
1+x^{2}\right) .
\end{equation*}%
Also, using (\ref{fourth mom}) we get%
\begin{equation*}
\left( S_{n,p,q}\left( \left( t-x\right) ^{4};x\right) \right) ^{1/2}\leq
A_{3}\left( 1+x^{2}\right) 
\end{equation*}%
and%
\begin{equation*}
S_{n,p,q}\left( \frac{\left( t-x\right) ^{2}}{\delta ^{2}};x\right)
^{1/2}\leq \frac{A_{4}}{\delta }\mathcal{O}\left( \beta _{p,q}\left(
n\right) \right) ^{1/2}\left( 1+x^{2}\right) ^{1/2}
\end{equation*}%
for $A_{3}>0$ and $A_{4}>0.$ So we have%
\begin{eqnarray*}
\left \vert S_{n,p,q}\left( f;x\right) -f\left( x\right) \right \vert  &\leq
&2\left( 1+\frac{1}{\beta _{p,q}\left( n\right) }\right) \Omega \left( f;1/%
\sqrt{\beta _{p,q}\left( n\right) }\right) \left( 1+x^{2}\right) \left \{
1+A_{2}\left( 1+x^{2}\right) \right.  \\
&&\left. +\frac{A_{4}}{\delta }\mathcal{O}\left( \beta _{p,q}\left( n\right)
\right) ^{1/2}\left( 1+x^{2}\right) ^{1/2}+A_{3}\left( 1+x^{2}\right) \frac{%
A_{4}}{\delta }\mathcal{O}\left( \beta _{p,q}\left( n\right) \right)
^{1/2}\left( 1+x^{2}\right) ^{1/2}\right \} 
\end{eqnarray*}%
Choosing $\delta =\beta _{p,q}\left( n\right) ^{1/2},$we obtain%
\begin{eqnarray*}
\left \vert S_{n,p,q}\left( f;x\right) -f\left( x\right) \right \vert  &\leq
&2\left( 1+\beta _{p,q}\left( n\right) \right) \Omega \left( f;1/\sqrt{\beta
_{p,q}\left( n\right) }\right) \left( 1+x^{2}\right) \left \{ 1+A_{2}\left(
1+x^{2}\right) \right.  \\
&&\left. +CA_{4}\left( 1+x^{2}\right) ^{1/2}+C_{1}A_{3}A_{4}\left(
1+x^{2}\right) ^{3/2}\right \} .
\end{eqnarray*}%
For $0<q<p\leq 1,$ $\beta _{p,q}\left( n\right) \leq 1.$ Hence we can write%
\begin{equation*}
\sup_{x\in \left[ 0,\infty \right) }\frac{\left \vert S_{n,p,q}\left(
f;x\right) -f\left( x\right) \right \vert }{\left( 1+x^{2}\right) ^{5/2}}\leq
A\Omega \left( f;1/\sqrt{\beta _{p,q}\left( n\right) }\right) ,
\end{equation*}%
where $A=4\left( 1+A_{2}+CA_{4}+C_{1}A_{3}A_{4}\right) ,$ hence the result
follows.
\end{proof}

\section{Voronovskaya Theorem for $S_{n,p,q}$}

Here we give Voronovskaya theorem for $S_{n,p,q}.$

\begin{theorem}
Let $0<q_{n}<p_{n}\leq 1$ such that $q_{n}\rightarrow 1,$ $p_{n}\rightarrow
1,$ $q_{n}^{n}\rightarrow a,$ $p_{n}^{n}\rightarrow b$ as $n\rightarrow
\infty .$ For any $f\in C_{2}^{\ast }\left[ 0,\infty \right) $ such that $%
f^{\prime }$, $f^{\prime \prime }\in C_{2}^{\ast }\left[ 0,\infty \right) $
we have 
\begin{equation*}
\lim_{n\rightarrow \infty }\left[ n\right] _{p_{n},q_{n}}\left[
S_{n,p_{n},q_{n}}\left( f;x\right) -f\left( x\right) \right] =\alpha
xf^{\prime }\left( x\right) +\left( \gamma x^{2}+x\right) f^{\prime \prime
}\left( x\right) 
\end{equation*}%
uniformly on any $\left[ 0,A\right] ,$ $A>0.$
\end{theorem}

\begin{proof}
Let $f,$ $f^{\prime },$ $f^{\prime \prime }\in C_{2}^{\ast }\left[ 0,\infty
\right) $ and $x\in \left[ 0,\infty \right) .$ By the Taylor formula we can
write%
\begin{equation}
f\left( t\right) =f\left( x\right) +f^{\prime }\left( x\right) \left(
t-x\right) +\frac{1}{2}f^{\prime \prime }\left( x\right) \left( t-x\right)
^{2}+h\left( t,x\right) \left( t-x\right) ^{2},  \label{eq5}
\end{equation}%
where $h\left( t,x\right) $ is Peano form of the remainder. The function $%
h\left( \cdot ,x\right) \in C_{2}^{\ast }\left[ 0,\infty \right) $ and for $%
n $ large enough $\lim_{t\rightarrow x}h\left( t,x\right) =0.$

Applying the operators (\ref{op}) to both sides of (\ref{eq5}) we get%
\begin{eqnarray*}
\left[ n\right] _{p_{n},q_{n}}\left[ S_{n,p_{n},q_{n}}\left( f;x\right)
-f\left( x\right) \right]  &=&\left[ n\right] _{p_{n},q_{n}}f^{\prime
}\left( x\right) S_{n,p_{n},q_{n}}\left( t-x;x\right) +\left[ n\right]
_{p_{n},q_{n}}f^{\prime \prime }\left( x\right) S_{n,p_{n},q_{n}}\left(
\left( t-x\right) ^{2};x\right)  \\
&&+S_{n,p_{n},q_{n}}\left( h\left( t,x\right) \left( t-x\right)
^{2};x\right) .
\end{eqnarray*}%
By the Cauchy-Schwarz inequality, we have%
\begin{equation}
S_{n,p_{n},q_{n}}\left( h\left( t,x\right) \left( t-x\right) ^{2};x\right)
\leq \sqrt{S_{n,p_{n},q_{n}}\left( h^{2}\left( t,x\right) ;x\right) }\sqrt{%
S_{n,p_{n},q_{n}}\left( \left( t-x\right) ^{4};x\right) }.  \label{k1}
\end{equation}%
Observe that $h^{2}\left( x,x\right) =0$ and $h^{2}\left( \cdot ,x\right)
\in C_{2}^{\ast }\left[ 0,\infty \right) .$ Then it follows from Theorem \ref%
{uniform theo} that%
\begin{equation}
\lim_{n\rightarrow \infty }S_{n,p_{n},q_{n}}\left( h^{2}\left( t,x\right)
;x\right) =h^{2}\left( x,x\right) =0  \label{k2}
\end{equation}%
uniformly with respect to $x\in \left[ 0,A\right] .$ Hence from (\ref{k1}), (%
\ref{k2}) and (\ref{r1}) we obtain%
\begin{equation}
\lim_{n\rightarrow \infty }\left[ n\right] _{p_{n},q_{n}}S_{n,p_{n},q_{n}}%
\left( h\left( t,x\right) \left( t-x\right) ^{2};x\right) =0.  \label{h1}
\end{equation}%
Then using (\ref{s1}), (\ref{s2}) and (\ref{h1}) we have%
\begin{eqnarray*}
\lim_{n\rightarrow \infty }\left[ n\right] _{p_{n},q_{n}}\left[
S_{n,p_{n},q_{n}}\left( f;x\right) -f\left( x\right) \right]  &=&f^{\prime
}\left( x\right) \lim_{n\rightarrow \infty }\left[ n\right]
_{p_{n},q_{n}}S_{n,p_{n},q_{n}}\left( t-x;x\right)  \\
&&+f^{\prime \prime }\left( x\right) \lim_{n\rightarrow \infty }\left[ n%
\right] _{p_{n},q_{n}}S_{n,p_{n},q_{n}}\left( \left( t-x\right)
^{2};x\right)  \\
&&+\lim_{n\rightarrow \infty }\left[ n\right]
_{p_{n},q_{n}}S_{n,p_{n},q_{n}}\left( h\left( t,x\right) \left( t-x\right)
^{2};x\right)  \\
&=&\alpha xf^{\prime }\left( x\right) +\left( \gamma x^{2}+x\right)
f^{\prime \prime }\left( x\right) ,
\end{eqnarray*}%
which is desired.
\end{proof}

\end{document}